\documentclass[11pt]{amsart}
\usepackage{graphicx}
\usepackage{mathpazo}
\usepackage{amssymb}
\usepackage{epstopdf}
\usepackage[cmtip,all]{xy}
\usepackage{hyperref}
\usepackage{mathrsfs}
\usepackage{enumerate}
\hypersetup{linktocpage}
\newtheorem{prop}{Proposition}[section]
\newtheorem{lemma}[prop]{Lemma}
\newtheorem{thm}[prop]{Theorem}

\usepackage[pdftex,paper=a4paper,portrait=true,textwidth=400pt,textheight=670pt,tmargin=4cm,marginratio=1:1]{geometry}

\usepackage{mathrsfs}
\usepackage{amsthm}
\usepackage{amssymb}
\usepackage{amsfonts}
\usepackage{tikz} 
\usepackage{mathtools}
\usepackage{setspace}
\usepackage{cancel}
\usepackage{chemarrow}
\usetikzlibrary{matrix,arrows} 
\usepackage{amscd}

\newtheorem{theorem}{Theorem}

\newtheorem{cor}[prop]{Corollary}

\newtheorem{assumption}{Assumption}
\theoremstyle{definition}

\newtheorem{remark}[prop]{Remark}
\newtheorem{defn}[prop]{Definition}

\newcommand{\X}{\mathfrak{X}}
\newcommand{\CC}{\mathbb{C}}
\renewcommand{\P}{\mathbb{P}}
\renewcommand{\L}{\mathcal{L}}

\newcommand{\tX}{\widetilde{X}}

\newcommand{\tL}{\widetilde{L}}

\newcommand{\tD}{\widetilde{D}}
\newcommand{\bP}{{\bf{P}}}
\newcommand{\I}{\mathcal{I}}

\newcommand{\tC}{\widetilde{C}}
\newcommand{\rC}{\mathfrak{C}}

\newcommand{\V}{\mathcal{V}}
\newcommand{\tpi}{\widetilde{\pi}}
\renewcommand{\O}{\mathcal{O}}
\newcommand{\Z}{\mathbb{Z}}
\newcommand{\Q}{\mathbb{Q}}
\newcommand{\F}{\mathcal{F}}

\newcommand{\G}{\mathcal{G}}
\newcommand{\K}{\mathcal{K}}
\newcommand{\FF}{\mathbb{F}}
\newcommand{\GG}{\mathbb{G}}
\renewcommand{\SS}{\mathbb{S}}
\newcommand{\Ext}{\mathcal{E}xt}
\newcommand{\supp}{\operatorname{supp}}
\newcommand{\M}{\mathcal{M}}

\newcommand{\hilb}{\operatorname{Hilb}}
\newcommand{\disp}{\displaystyle}
\newcommand{\ex}{\operatorname{ext}}

\newcommand{\Home}{\mathcal{H}om}
\newcommand{\Hom}{\operatorname{Hom}}
\newcommand{\ch}{\mathsf{ch}}
\newcommand{\ob}{\operatorname{Ob}}
\newcommand{\ext}{\operatorname{Ext}}

\newcommand{\W}{\mathfrak{W}}
\newcommand{\Y}{\mathfrak{Y}}
\newcommand{\quot}{\operatorname{Quot}}
\newcommand{\A}{\mathbb{A}}
\newcommand{\fA}{\mathfrak{A}}

\newcommand{\fC}{\mathfrak{C}}

\newcommand{\cX}{\mathcal{X}}

\title[Donaldson-Thomas Invariants of 2-dimensional sheaves]{Donaldson-Thomas Invariants of 2-Dimensional sheaves inside threefolds and modular forms}
\author[Amin Gholampour and Artan Sheshmani]{Amin Gholampour and Artan Sheshmani}
\date{\today}                                   
\begin{document}
\maketitle
\begin{abstract}
Motivated by the S-duality conjecture, we study the Donaldson-Thomas invariants of the 2-di\-men\-sion\-al Gieseker stable sheaves on a threefold. These sheaves are supported on the fibers of a nonsingular threefold $X$ fibered over a nonsingular curve. In the case where $X$ is a $K3$ fibration, we express these invariants in terms of the Euler characteristic of the Hilbert scheme of points on the $K3$ fiber and the Noether-Lefschetz numbers of the fibration. We prove that a certain generating function of these invariants is a vector modular form of weight $-3/2$ as predicted in S-duality.

\end{abstract}
\section{Introduction}
 \label{sec:intro}
 \subsection{Overview}
We study the invariants virtually counting the configurations of a number of points and a vector bundle supported on the members of a system of divisors inside a nonsingular threefold. One of our motivations is that these invariants have been studied by the physicists  \cite{DM07,a78,a95,a100, a136} as a set of supersymmetric BPS invariants associated to D4-D2-D0 systems. By S-duality conjecture, the generating series of these invariants are expected to be modular. 

In this paper we interpret these invariants by means of the moduli spaces of Gieseker stable coherent sheaves with 2-dimensional support inside a nonsingular threefold $X$. Another motivation for considering these moduli spaces is to find a sheaf-theoretic analog of the formulas proven in \cite{a90} that relate the Gromov-Witten invariants of a threefold to the Gromov-Witten invariants of a system of its divisors.

\subsection{Moduli spaces of 2-dimensional sheaves}\label{sec:intro}
Let $X$ be a nonsingular projective threefold over $\CC$ with a fixed polarization $L$. For a given nonzero effective irreducible divisor class $F\in \text{Pic}(X)$, we fix a Chern character vector \begin{equation} \label{ch}\ch=(0,rF,\gamma,\ch_3)\in \oplus_{i=0}^3 H^{2i}(X,\Q),\end{equation} with $r>0$.  If $\F$ is a coherent sheaf with $\operatorname{ch}(\F)=\ch$ then, $\F$ is supported on some divisor with the numerical class $r'F$ where $r' \mid r$. We always assume that any Gieseker $L$-semistable sheaf $\F$ with $\operatorname{ch}(\F)=\ch $ is stable.

We consider the moduli space $\M=\M^L(X,\ch)$ of Gieseker $L$-semistable sheaves with Chern character $\ch$.  By our assumption, $\M$ is projective and the geometric points correspond to the isomorphism classes of stable (hence pure) 2-dimensional sheaves with Chern character $\ch$. It is proven in \cite{a20} that $\M$ admits a perfect obstruction theory if 

 $$\ext^3(\F,\F)_0=0$$  for any geometric point $\F \in \M$, where the index $0$ indicates the trace free part of $\ext^3(\F,\F)$. 
In this case, Thomas constructs (Theorem \ref{thm:trunc1}) a natural perfect obstruction theory $E^\bullet \to L^\bullet_\M$ such that $$h^i({E^\bullet}^\vee)_\F\cong \ext^{i+1}(\F,\F)$$ for $i=0,1$. By \cite{a2,a14} one obtains a virtual cycle $$[\M]^{vir}=[\M,E^\bullet]^{vir} \in A_{\mathrm{vd}}(\M)$$ where $\mathrm{vd}=\ex^1(\F,\F)-\ex^2(\F,\F)$ is the virtual dimension of $\M$.

In this paper, we only consider the case $\mathrm{vd}=0$. The Donaldson-Thomas invariants are then defined by  \begin{equation} \label{DT} DT(X,\ch)=\deg( [\M]^{vir}) \in \Z.\end{equation} 
When $X$ is a Calabi-Yau threefold, the obstruction theory above is symmetric (and hence $\mathrm{vd}=0$), and by Behrend's celebrated result $DT(X, \ch)=e(\M,\nu_{\M})$ where $\nu_{M}$ is the Behrend's constructible function (\cite{a1}). %

Besides Calabi-Yau threefolds, as we will see in Section \ref{sec:fibration} (see Proposition \ref{prop:vd}), another important situation in which  $\mathrm{vd}=0$ occurs is when   
\begin{enumerate}[($\star$i)] 
\item the nonsingular projective threefold $X$ admits a surjective morphism \begin{equation}\label{equ:fibration} \pi:X\to C \end{equation} with irreducible fibers to a nonsingular projective genus $g$ curve $C$,
\item $K_X\cong \pi^*\K$ for some $\K \in \text{Pic}(C)$,
\item $F$ in \eqref{ch} is the class of a fiber,
\item any Gieseker $L$-semistable sheaf $\F$ with $\operatorname{ch}(\F)=\ch$ as in \eqref{ch} is stable.
\end{enumerate}

This paper is devoted to a detailed study of the moduli space $\M=\M^L(X,\ch)$ and the invariants $DT(X,\ch)$ in the latter situation. $\M$ can have several irreducible components of different dimensions. In the case that \eqref{equ:fibration} is a smooth morphism, we find explicit formulas for the restrictions of $[\M]^{vir}$ to these components.

\subsection{Noether-Lefschetz numbers}
In case the fibration \eqref{equ:fibration} is a smooth $K3$ fibration, \cite{a90} defines the Noether-Lefschetz numbers $NL^\pi_{h,\gamma}$ as intersection numbers in the moduli space of quasi-polarized $K3$ surfaces, and relates them to the invariants defined by Kudla-Millson \cite{a117} and Borcherds \cite{a116, a125}. Here $X$ does not need to be projective, and the line bundle $L$ on $X$ is only required to give a quasi-polarization when restricted to each fiber (see \cite[Section 1.2]{a90}). 

 Let  $\ell=F \cdot L^2$, and $\gamma^{PD} \in H_2(X,\Z)^\pi$ with $L \cdot \gamma =d$.
Informally, $NL^\pi_{h,\gamma}$ is the number of the fibers $i:S\hookrightarrow X$ of $\pi$ for which there exists a $(1,1)$ class $\beta \in H^2(S,\Z)$ such that $$\beta^2=2h-2 \;\;\text{and}\;\;  i_*\beta^{PD}=\gamma^{PD}.$$ Here $PD$ stands for the Poincar\'{e} dual. It is proven in \cite{a90} that $NL^\pi_{h,\gamma}$ vanishes if $h>1+d^2/2\ell$, and 
\begin{equation} \label{symmetry}
NL^\pi_{h,\gamma}=NL^\pi_{h+d+\ell/2, \gamma+F\cdot L}.
\end{equation}

For $\disp NL^{\pi}_{h,d}=\sum_{\tiny \begin{array}{c}   \gamma^{PD} \in H_2(X,\Z)^\pi \\L\cdot \gamma=d \end{array}}NL^{\pi}_{h,\gamma},$ define
$$\Phi^{\pi}_d(q)=q^{1+d^2/2\ell}\sum_{h\in \Z} NL^{\pi}_{h,d} q^{-h},$$ and set $$\Phi^{\pi}(q)=\sum_{d=0}^{\ell-1}\Phi^{\pi}_d(q)v_d \in\CC[[q^{1/2\ell}]]\otimes \CC[\Z/\ell\Z].$$  Then, $\Phi^{\pi}(q)$ is a vector valued modular form of weight 21/2 \cite{a90, a125, a116}. 

\subsection{Main results}
Our main application is the case when the fibration \eqref{equ:fibration} is a smooth $K3$ fibration. The condition ($\star$ii) above is then immediate. Let $$\mathrm{k}=\deg(\K)-2g+2.$$
We consider a slightly more general situation in which $X$ is only locally projective over $C$; $\M(X,\ch)$ and $DT(X,\ch)$ can still be defined (see Remark \ref{rem:nonproj}). Then, we express $DT(X,\ch)$ in terms of the Euler characteristic of the Hilbert scheme of points on a $K3$ surface and the Noether-Lefschetz numbers:

\begin{theorem} \label{thm:main formula}
Let $\pi:X\to C$ be a smooth $K3$ fibration with the fiber class $F$. For the Chern character vector $\ch$ as in \eqref{ch} satisfying ($\star$iv) above,  we have 
\begin{align*} DT(X, \ch)=&\sum_{h\in \Z} e(\operatorname{Hilb}^{h-r\ch_3-r^2}(K3))\cdot NL^\pi_{h,\gamma}\\&-\mathrm{k}\delta_{0, \gamma} \cdot e(\operatorname{Hilb}^{1-r\ch_3 -r^2}(K3)),
\end{align*}
where $\delta_{0, \gamma}$ is the Kronecker delta function. 
 \end{theorem}
The proof of Theorem \ref{thm:main formula} follows the same ideas as the proof of Theorem 1 in \cite{a90}, and is given in Section \ref{sec:K3}. Using \eqref{symmetry} and the formula above, we will deduce a curious symmetry among the invariants $DT(X,\ch)$ (see Corollary \ref{cor:comp}).

Now we set $r=1$ in $\ch$ and take the $K3$ fibration $X$ above to be projective, but we allow $X$ to have finitely many nodal singular fibers. We will prove that a certain generating function of the invariants $DT(X,\ch)$, for which $\ch_0=0$ and $\ch_1=F$ are fixed and $\ch_2, \ch_3$ are summed over, is a vector modular form of weight $-3/2$. This confirms the prediction from $S$-duality mentioned above (see \cite[Section 4]{GST14}).

More precisely, let $$P_{d,c}(m)=(\ell/2)m^2+dm+c $$ be a degree 2 polynomial in $m$ with $c,d\in \Z$.  
We let $$DT_L(X,P_{d,c})=\sum_{\tiny \begin{array}{c} \gamma, \ch_3 \; \text{ s.t.} \\ \ch=(0,F,\gamma,\ch_3)\\P_L(\ch)=P_{d,c} \end{array}}DT(X,\ch),$$ where  $P_L(\ch)$ is the Hilbert polynomial associated to the Chern character vector $\ch$. Note that since $r=1$, for any choice of $\gamma$ and $\ch_3$, the Chern character vector $\ch$ satisfies ($\star$iv) above.
Define the generating series $$Z_d(X,q)=q^{1+d^2/2\ell}\sum_{c\in \Z}DT_L(X,P_{d,c})q^{-c},$$  $$Z(X,q)=\frac{\mathrm{k}v_0}{\Delta(q)}+\sum_{d=0}^{\ell-1}Z_d(X,q)v_d \in \CC[[q^{1/2\ell}]]\otimes \CC[\Z/\ell\Z]$$ where $\Delta(q)=q\prod_{n\ge 1}(1-q^n)^{24}$ is the discriminant cusp form of weight $12$.  Note that tensoring by $L^{\pm 1}$ induces isomorphisms of the moduli spaces $\M(X,\ch)$, so the generating series $Z_d(X,q)$ and $Z_{d\pm \ell}(X,q)$ only differ by a shift in the power of $q$. We associate to $X$ a (not necessarily projective) $K3$ fibration $\tpi:\tX \to \tC$ with nonsingular fibers. For this, we take a double cover of $X$ branched over the singular fibers, and then we resolve the conifold singularities of the resulting threefold  (see Section \ref{sec:conifold} for more details).
Theorem \ref{thm:main formula} can be applied to $\tX$. We then use degeneration techniques to express the invariants of $X$ in terms of the invariants of $\tX$. As a result, we obtain the following theorem:
\begin{theorem} \label{thm:K3} Let $\pi:X\to C$ be a $K3$ fibration with finitely many of the fibers having nodal singularities. 
Then, $$Z(X,q)=\frac{\Phi^{\tpi}(q)}{2\Delta(q)}.$$ 
\end{theorem}
The proof of Theorem \ref{thm:K3} is given in Section \ref{sec:conifold}. To our knowledge, the result of Theorem \ref{thm:K3} is the only instance where the full generating series of such DT invariants is proven to be given by a vector valued modular form.

The basic example of $X$ in Theorem \ref{thm:K3} is the Lefschetz pencil of quartics $\pi:X \subset \P^3\times \P^1\to \P^1$ (generic divisor of type $(4,2)$) which is a Calabi-Yau threefold, and via the projection to $\P^1$, is realized as a $K3$ fibration with 216 nodal fibers. In this case $\mathrm{k}=2$, and  $\Phi^{\tpi}(q)$ is explicitly evaluated in \cite{a90} for the nonsingular models of the Lefschetz pencil of quartics, and hence $Z(X,q)$ is completely known by Theorem \ref{thm:K3}.

\section*{Acknowledgment}
We would like to thank Kai Behrend, Patrick Brosnan, Jim Bryan, Daniel Huybrechts, Jun Li, Davesh Maulik, Gregory Moore, Yukinobu Toda and particularly Richard Thomas for many helpful discussions. 

A. G. was partially supported by NSF grant DMS-1406788. A. S. was partially supported by World Premier International Research Center Initiative (WPI initiative), MEXT, Japan, during the time that this work was in progress. Furthermore, A. S.  would like to thank the University of British Columbia, Max Planck Institut f\"{ur} Mathematik, Isaac Newton Institute for Mathematical Sciences in University of Cambridge, Kavli IPMU, Center for quantum geometry of moduli spaces at Aarhus University, and Center for mathematical sciences and applications at Harvard University for their help and support.

\section{Threefolds fibered over curves} \label{sec:fibration}
Let $\pi:X \to C$  be as in \eqref{equ:fibration} and the conditions ($\star$i)-($\star$iv) in Section \ref{sec:intro} are satisfied.  By Bertini's theorem general fibers of $\pi$ are nonsingular. In other words, we allow finitely many fibers to be possibly singular or even non-reduced. 
A typical example for such an $X$ is the Enriques Calabi-Yau threefold considered by \cite{a120} in the context of Gromov-Witten theory; it can be realized as a $K3$ fibration over $\P^1$ with 4 doubled Enriques surfaces. 

We consider $\M=\M^L(X,\ch)$, the moduli space of Gieseker $L$-stable coherent sheaves with the Chern character $\ch$. Any geometric point of $\M$ corresponds to the isomorphism  class of a coherent sheaf $\F$ which is supported on the fiber(s) of $\pi$.  In fact we have: 

\begin{lemma} \label{reduced}
For any geometric point $\F \in \M$, the support of $\F$ is reduced and connected. 
\end{lemma}
\begin{proof}
The support of $\F$ is connected since $\F$ is stable. To see the support of $\F$ is reduced, denote by $S$ the support of $\F$ with the reduced induced structure, and let $I=\O_X(-S)$ be the ideal sheaf of $S$ in $X$. We have $\operatorname{ch}(I|_{S})=1$, so since $\F$ is supported on $S$, we see that $\F$ and $\F\otimes I$ have the same Hilbert polynomial. Tensoring the short exact sequence $0\to I\to \O_X\to \O_S\to 0$ by $\F$, we get the exact sequence $$\F \otimes I\to \F \to \F|_S\to 0.$$ By the stability of $\F$ and $\F\otimes I$, the first map in the sequence above is either an isomorphism or the zero map (see \cite{a9}[Proposition 1.2.7]). But the former is not possible (otherwise $\F|_S=0$), so we conclude that the second map in the sequence above is an isomorphism, and this finishes the proof.  
\end{proof}

\begin{prop} \label{prop:vd}
$\M$ admits a perfect obstruction theory with $\mathrm{vd}(M)=0$.
\end{prop}
\begin{proof}
By Lemma \ref{reduced} and ($\star$ii) in Section \ref{sec:intro}, for any $\F \in \M$ we have  $\F\otimes K_X \cong \F$. Serre duality then implies that $$\operatorname{Ext}^3(\F,\F)\cong \operatorname{Hom}(\F,\F)^\vee \cong \CC \quad \text{and} \quad \operatorname{Ext}^2(\F,\F)\cong \operatorname{Ext}^1(\F,\F)^\vee.$$ 
Therefore, $\ext^3(\F,\F)_0=0$ over $\mathcal{M}$. Proposition follows from \cite[Theorem 3.30]{a20}.  
\end{proof}

\textbf{Notation.} We denote by $p:X\times \M\rightarrow \M$ the projection, and let $\rho: \mathcal{M}\to C$ be the natural morphism which sends the stable sheaf $\F$ with support $S$ to $\pi(S)\in C$. The morphism $\rho$ is well-defined by Lemma \ref{reduced}. 
Then the universal sheaf $\FF$ \footnote{$\FF$ may exist only as a twisted sheaf (see \cite[Section 3]{a20} for a relevant discussion).} is the push forward of a rank $r$ torsion-free sheaf $\mathbb{G}$ supported on the codimension 1 closed subscheme $$\SS:=X {_\pi}\times_\rho \M \xhookrightarrow{i} X\times M.$$ $\SS$ is the pullback of the diagonal in $C\times C$, and hence $\O_\SS(-\SS)\cong i^*\, p^* \, \rho^* \;\omega_C$.
\begin{thm} (\cite{a20}) \label{thm:trunc1}
The perfect obstruction theory over $\mathcal{M}$ mentioned in Proposition \ref{prop:vd} is given by the following morphism in the derived category:
\begin{equation*} 
E^{\bullet}:=\left(\tau^{[1,2]}Rp_{*}(R\mathscr{H}om_{X\times\mathcal{M}}(\mathbb{F},\mathbb{F})\right)^{\vee}[-1]\rightarrow\mathbb{L}^{\bullet}_{\mathcal{M}}.
\end{equation*} \qed
\end{thm} 

The corresponding DT invariant $DT(X,\ch)$ is defined as in \eqref{DT}.

\subsection{Smooth fibrations} \label{sec:smooth}

In this section we specialize to the case where the morphism \eqref{equ:fibration} is smooth and the conditions ($\star$i)-($\star$iv) in Section \ref{sec:intro} are satisfied. Then the fibers of $\pi$ are nonsingular projective surfaces with trivial canonical bundles. 
\begin{defn} \label{defn:comps}
We call a connected component $\M_{c}$ of $\M$ a \emph{type I component} if $\rho(\M_c)=C$, otherwise we call $\M_c$ a \emph{type II component}. A type II component is called \emph{isolated} if it ismorphic to a moduli space of torsion free sheaves on a nonsingular fiber of $\pi$. 
We usually denote a type I component by $\M_0$ and an isolated type II component by $\M_{iso}$. 
\end{defn}

Mukai in \cite{a113} proves that the moduli space of stable sheaves on nonsingular projective surfaces with trivial canonical bundles is nonsingular. This immediately implies that 
\begin{lemma} \label{lem:mukai}
If $\M_c$ is a type I or an isolated type II component of $\M$ then $\M_c$ is nonsingular. \qed
\end{lemma}

  We know that for any fiber $S$ of $\pi$ and any coherent sheaf $\G$ supported on $S$, $\ext^3_S(\G,\G)=0$. This essentially implies that a type I component $\M_0$ of  $\M$ admits a perfect $\rho$-relative obstruction theory (in the sense of \cite{a2}). 
\begin{thm}\label{thm:trunc2} (\cite{a20,a10})
There is a $\rho$-relative obstruction theory over a type I component $\M_0$ given by the following morphism in the derived category:
\begin{equation*}
 \left(\tau^{\geq 1}Rp_*\,i_*R\mathscr{H}om_{X {_\pi}\times_\rho \M_0}(\mathbb{G},\mathbb{G})\right)^{\vee}[-1]\rightarrow\mathbb{L}^{\bullet}_{\M_0/C}.
\end{equation*} \qed
\end{thm}

\begin{prop}\label{prop:absolute}
The $\rho$-relative obstruction theory over a type I component $\M_0$ given by Theorem \ref{thm:trunc2} induces a perfect obstruction theory $F_0^{\bullet}\rightarrow \mathbb{L}^{\bullet}_{\M_0}$ over $\M_0$.
\end{prop}
\begin{proof}
Composing the morphism in the derived category in Theorem \ref{thm:trunc2} with the second morphism in the exact triangle 
 $\mathbb{L}^{\bullet}_{\M_0}\rightarrow\mathbb{L}^{\bullet}_{\M_0/C}\rightarrow \rho^* \omega_{C}[1],$ we obtain a morphism in the derived category $$\lambda:\left(\tau^{\geq 1}Rp_{*}\,i_*R\mathscr{H}om(\mathbb{G},\mathbb{G})\right)^{\vee}[-1]\rightarrow \rho^* \omega_{C}[1].$$ Now the two exact triangles below induce a morphism $g$ (the left vertical map) in the derived category.
\begin{equation*}\label{cone}
\begin{tikzpicture}
back line/.style={densely dotted}, 
cross line/.style={preaction={draw=white, -, 
line width=6pt}}] 
\matrix (m) [matrix of math nodes, 
row sep=2em, column sep=2.5em, 
text height=1.5ex, 
text depth=0.25ex]{ 
\operatorname{Cone}(\lambda)[-1]&\left(\tau^{\geq 1}Rp_*\,i_*R\mathscr{H}om(\mathbb{G},\mathbb{G})\right)^{\vee}[-1]&\rho^* \omega_{C}[1]\\
\mathbb{L}^{\bullet}_{\M_0}&\mathbb{L}^{\bullet}_{\M_0/C}&\rho^* \omega_{C}[1]\\};
\path[->]
(m-1-1) edge node [right] {$g$}(m-2-1)
(m-1-1) edge (m-1-2)
(m-1-2) edge node [above] {$\lambda$}(m-1-3)  
(m-1-2) edge (m-2-2)
(m-2-1) edge (m-2-2)
(m-2-2) edge (m-2-3)
(m-1-3) edge node [right] {$id$} (m-2-3);
\end{tikzpicture}
\end{equation*}
Define $F_0^{\bullet}:=\operatorname{Cone}(\lambda)[-1]$, then comparing the induced long exact sequences of cohomologies, one proves that $g: F_0^\bullet \to \mathbb{L}^{\bullet}_{\M_0}$ defines a perfect deformation-obstruction theory for $\M_0$. 
\end{proof}

We use the following proposition to compare the obstruction theories given by Theorems \ref{thm:trunc1} and \ref{thm:trunc2} on a type I component $\M_0$:

\begin{prop}\label{exact-rhom}

There exists the following exact triangle in the derived category of $ \M_0$:
\begin{align*}
&
\tau^{\ge 1}Rp_{*}\,i_*\; R\mathscr{H}om_{X{_\pi}\times_\rho \M_0}(\mathbb{G},\mathbb{G}) \rightarrow \tau^{[1,2]} Rp_{*}\;R\mathscr{H}om_{X\times\mathcal{M}_0}(\mathbb{F},\mathbb{F})\notag\\
&
\rightarrow \tau^{\le 2}Rp_{*}\, i_*\;R\mathscr{H}om_{X{_\pi}\times_\rho \M_0}(\mathbb{G},\mathbb{G}(\SS))[-1] 
\end{align*}
\end{prop}
\begin{proof} 
Applying the functor $R\mathscr{H}om(-,\mathbb{G})$ to the exact triangle
$$\mathbb{G}(-\SS)[1]\rightarrow L{i^*}i_*\mathbb{G}\rightarrow \mathbb{G},$$ 
and using the adjoint property $L{i^*_{\SS}}\vdash i_{\SS*}$, we get the exact triangle
\begin{align}\label{triang}
&
i_{*}R\mathscr{H}om(\mathbb{G},\mathbb{G})\rightarrow R\mathscr{H}om(\mathbb{F},\mathbb{F})\rightarrow i_{*}R\mathscr{H}om(\mathbb{G},\mathbb{G}(\SS))[-1]. 
\end{align}
The result follows by applying $Rp_{*}$ and truncating the resulting exact triangle. Note that after applying $Rp_{*}$ to \eqref{triang} the terms of the resulting exact triangle from left to right are respectively concentrated in degrees $[0,2]$, $[0,3]$, and $[1,3]$. Moreover, the $0$th cohomologies of the first two terms, and the $3$rd cohomologies of the last two terms are isomorphic. As a result, after truncation we still get an exact triangle.
\end{proof}

Theorem \ref{thm:trunc1} and Proposition \ref{prop:absolute} give the virtual cycles $[\mathcal{M}_0,E^{\bullet}]^{vir}$ and $[\mathcal{M}_0,F^{\bullet}_0]^{vir}$ (see \cite{a2}). The relation between these two cycles is given by  

\begin{prop}\label{decompos} Let $\M_0$ be a type I component, and $\ob_0$ be the locally free sheaf on $\M_0$ obtained by restricting $\mathcal{E}xt^{1}_{p\circ i}(\mathbb{G},\mathbb{G}(\mathbb{S}))$ to $\M_0$ then \begin{equation*} 
[\mathcal{M}_0, E^{\bullet}]^{vir}=c_{top}(\ob_0) \cap[\mathcal{M}_0, F_0^{\bullet}]^{vir}.
\end{equation*}
\end{prop}
\begin{proof}
By Proposition \ref{exact-rhom}, $\ob_0$ fits into the short exact sequence of the vector bundle stacks over $\M_0$ (see \cite{a2}):

$$h^1/h^0({F_0^\bullet}^\vee)\xrightarrow{i} h^1/h^0({E^\bullet}^\vee)\xrightarrow{j} \ob_0.$$ Denote by $0_{\ob_0}$ and $0_{{E^\bullet}^\vee}$ the zero section embeddings of $\M_0$. Then this gives the Cartesian diagram 

\begin{equation*} 
\xymatrix{
    h^1/h^0({F_0^\bullet}^\vee)  \ar[r]^i \ar[d] & h^1/h^0({E^\bullet}^\vee)  \ar[d]^j\\
\M_0  \ar[r]^{0_{\ob_0}} & \ob_0}
\end{equation*}

 Now if $\mathfrak{C}_{\M_0}$ is the intrinsic normal cone of $\M_0$, from the diagram above, we get (see \cite{a111}, \cite{a2})
\begin{align*}
[\mathcal{M}_0,E^{\bullet}]^{vir}&=0_{{E^\bullet}^\vee}^{!}[\mathfrak{C}_{\mathcal{M}_0}]=(i\circ 0_{{F_0^\bullet}^\vee})^{!}[\mathfrak{C}_{\mathcal{M}_0}]
=0_{{F_0^\bullet}^\vee}^{!}\circ i^{!}[\mathfrak{C}_{\mathcal{M}_0}]\\
&=0_{{F_0^\bullet}^\vee}^{!}\circ 0_{\ob_0}^{!}[\mathfrak{C}_{\mathcal{M}_0}]
=c_{top}(\ob_0)\cap 0_{{F_0^\bullet}^\vee}^{!} [\mathfrak{C}_{\mathcal{M}_0}]\\
&=c_{top}(\ob_0)\cap [\mathcal{M}_0,F_0^{\bullet}]^{vir}. 
\end{align*}
\end{proof}

For our later application we need to express $c_{top}(\ob_0)$ in terms of the top Chern class of the tangent sheaf of $\M_0$, where $\ob_0$ and $\M_0$ are as in Proposition \ref{exact-rhom}. The component $\M_0$ is nonsingular (see Lemma \ref{lem:mukai}) with the relative tangent sheaf $\mathcal{T}_{\M_0/C}=\mathcal{E}xt^{1}_{p\circ i}(\mathbb{G},\mathbb{G})$ (see \cite[Corollary 4.5.2]{a9} and \cite{a138}). 
\begin{lemma}\label{ctop}
Suppose that the dimension of the type I component $\M_0$ is $n$. Then we have the following relation in $H^{*}(\M_0,\Q)$:
\begin{equation*}
c_{top}(\ob_0)= c_{n}(\mathcal{T}_{\M_0/C})+A \cdot \rho^*(B).
\end{equation*}
for some $A \in A^{n-1}(\M_0)$ and $B \in A^1(C)$.
\end{lemma}
\begin{proof}
Grothendieck-Riemann-Roch formula gives: 
\begin{align*}
&\operatorname{ch}\bigg(\sum_{j=0}^2(-1)^j\Ext^j_{p \circ i}(\mathbb{G},\mathbb{G}(\mathbb{S}))\bigg)\\&=p_* \, i_* \bigg(\operatorname{ch}(\mathbb{G})^{\vee}\cdot \operatorname{ch}(\mathbb{G})\cdot\operatorname{ch}(\mathcal{O}_{\SS}(\mathbb{S}))\cdot \operatorname{td}(X {_{\pi}}\times_\rho \M_0) \bigg)\\&=(1+c_1(\rho^* T_C))\cdot \operatorname{ch}\bigg(\sum_{j=0}^2(-1)^j\Ext^j_{p \circ i}(\mathbb{G},\mathbb{G}))\bigg)
\end{align*}
Note that $\operatorname{ch}(\Ext^i_{p\circ i}(\mathbb{G},\mathbb{G}(\mathbb{S})))$ and $\operatorname{ch}(\Ext^i_{p\circ i}(\mathbb{G},\mathbb{G}))$ for $i=0,2$ are the pull back of classes from $A^*(C)$ by the arguments similar to one given in the paragraph after this lemma. The lemma is proven by an inductive argument on $n$.
\end{proof}

To find the relation between the virtual cycle $[\M_0,F_0^\bullet]^{vir}$ and $[\M_0]$, we note that the obstruction sheaf 
$$h^{1}({F_0^{\bullet}}^{\vee})\cong \Ext^2_{p\circ i}(\mathbb{G},\mathbb{G})$$ 
is an invertible sheaf on $\M_0$. In fact the trace map defines an isomorphism 
$$\Ext^2_{p\circ i}(\mathbb{G},\mathbb{G})\xrightarrow{tr}R^{2}(p\circ i)_*\mathcal{O}_{\SS},$$
since by Nakayama lemma it is enough to show that $tr$ gives isomorphism on the level of fibers over the closed points of $\M_0$, and (fiberwise) the trace map is the Serre duality isomorphism. By Verdier duality and  $\omega_{p\circ i}\cong (p\circ i)^* \rho^* (\K\otimes T_C)$ we get
\begin{equation} \label{K} R^{2}(p\circ i)_*\mathcal{O}_{\SS} \cong ((p\circ i)_*\omega_{p \circ i})^{\vee} \cong \rho^* (\K\otimes T_C)^\vee.\end{equation}
Using \cite[Proposition 5.6]{a2}, we have proven
\begin{prop}\label{decompos2} Suppose that $\M_0$ is a type I component. Then we have
$$[\M_0,F_0^\bullet]^{vir}=\rho^* c_1( \K^\vee \otimes \omega_C)\cap [\M_0] .$$ \qed
\end{prop}

\begin{remark}
In more technical terms $[\M_0]=[\M_0,F^\bullet_0]_{red}^{vir}$, where the latter is \emph{the reduced virtual class} obtained from the obstruction theory $F_0^\bullet$ (see \cite{a84, a137, a90}). 
\end{remark}

We finish this section by studying what happens when we restrict the virtual cycle $[\M,E^\bullet]^{vir}$ to an isolated type II component $\M_{iso}$. 
\begin{prop} \label{prop:decompos3} Suppose that $\M_{iso}$ is an isolated type II component of $\M$, and let $\mathcal{T}_{iso}$ denote the tangent sheaf of $\M_{iso}$. Then $$[\M,E^\bullet]^{vir}|_{\M_{iso}}=c_{top}(\mathcal{T}_{iso})\cap [\M_{iso}].$$  
\end{prop}
\begin{proof}
 Applying $Rp_{*}$ to the exact triangle \eqref{triang}  in the proof of Proposition \ref{exact-rhom} (with $\M_0$ replaced by $\M_{iso}$) and taking cohomology, we get the following exact sequence on $\M_{iso}:$
$$\cdots \to \Home_{p\circ i}(\GG,\GG)\to \Ext^2_{p\circ i}(\GG,\GG)\to \Ext^2_{p}(\FF,\FF)\to \Ext^1_{p\circ i}(\GG,\GG)\to 0.$$
Note that $\GG\cong \GG(\SS)$ in this case. Since $\M_{iso}$ is a type II isolated component, by definition, we obtain the isomorphisms of the locally free sheaves
$$\Ext^1_{p\circ i}(\GG,\GG)\to \Ext^1_{p }(\FF,\FF)\cong \mathcal{T}_{iso},$$ from which we conclude that the first map in the sequence above is injective and hence an isomorphism. Hence, the exact sequence above implies that 
$$h^{1}({E^{\bullet}}^{\vee})\mid_{\M_{iso}}\cong \Ext^2_{p}(\FF,\FF) \cong \Ext^1_{p\circ i}(\GG,\GG).$$ 
Now the proposition follows from \cite[Proposition 5.6]{a2}. 
\end{proof}

\subsection{Smooth $K3$ fibrations} \label{sec:K3}
In this section we assume that the morphism \eqref{equ:fibration} is smooth and the fibers are $K3$ surfaces, and the conditions ($\star$iii)-($\star$iv) in Section \ref{sec:intro} are satisfied. Note that in this case, the condition ($\star$ii) is automatically satisfied.
Let $i:S\hookrightarrow X$ be the inclusion of the fiber of $\pi$ over a closed point $p\in C$, and suppose that a stable sheaf $\F \in \M$ is supported on $S$. Then $\F=i_*\G$ for a stable rank $r$ torsion-free sheaf $\G$ on $S$ with $c_1(\G)=\beta$ and $c_2(\G)=\tau$. It is easy to see that \begin{equation} \label{c-ch} i_*\beta^{PD}=\gamma^{PD} , \quad \quad (i_*(\beta^2/2-\tau)^{PD})^{PD}=\ch_3(\F)=\ch_3\end{equation} where $PD$ is the Poincar\'{e} dual. 
  Let \begin{equation} \label{mv} v=(r,\beta,\beta^2/2-\tau+r)\end{equation} be the corresponding Mukai vector and $\M(S,v)$ be the moduli space of semistable sheaves on $S$ with Mukai vector $v$. By our assumption, $\M(S,v)$ contains no strictly semistable sheaves, and hence it is nonsingular of dimension
 $$2-\int_S v \cdot v^\vee=2r\tau-(r-1)\beta^2-2(r^2-1).$$
The moduli space $\M(S,v)$ has been thoroughly studied \cite{a79,a113, a139}. The following result has been proven  in \cite{a113} and \cite[Section 6]{a9}: 

\begin{prop} \label{prop:def-inv} Let $S$ be a $K3$ surface and $v$ a primitive Mukai vector as in (\ref{mv}). Then $\M(S,v)$ is deformation invariant to $\hilb^n(S)$, the Hilbert scheme of $n$ points in $S$, where $$n=\beta^2/2- r\ch_3-r^2+1.$$ 
In particular, $e(\M(S,v))=e(\hilb^n(S)).$ \qed
\end{prop} 

\begin{remark} \label{rem:nonproj}
For our later use, we need to extend the construction of $\M(X,\ch)$ to the case where the $K3$-fibration $\pi:X\to C$ is possibly not projective. We will consider the case where there are finitely many $K3$ fibers $S$ of $\pi$ over which $L$ restricts to a quasi-polarization (see \cite{a90}). We assume that there is a finite open affine cover $\{U_j\}$ of $X$ over $C$ such that $\pi_{U_j}$ is projective and maps $U_j$ to an open subset of $C$. In this case, for the compactly supported Chern character vector $\ch$ as in \eqref{ch} (with $F$ the fiber class) the moduli spaces $\M(U_i,\ch)$ can be constructed using the fiberwise polarizations. $\M(U_i,\ch)$ and $\M(U_j,\ch)$ are canonically isomorphic over the overlaps $U_i \cap U_j$, and so they can be patched together to give a proper scheme $\M=\M(X,\ch)$ over $C$. The perfect obstruction theories and the virtual cycles over $\M(U_i,\ch)$ constructed in the last section are also glued together to give the corresponding virtual cycles over $\M$; we define $DT(X,\ch)$ as in \eqref{DT}, using the properness of $\M$.
\end{remark} 

Let $$\V:R^2\pi_*(\Z)\to C$$ be the rank 22 local system determined by the $K3$-fibration $\pi$. Let $\mathcal{H}^{\V}$ denote the $\pi$-relative moduli space of Hodge structures as in \cite[Section 1.4]{a90}.\footnote{In \cite[Section 1.4]{a90} this is denoted by $\mathcal{M}^{\V}$.} There exists a section map $\sigma:C\rightarrow \mathcal{H}^{\V}$ which is determined by the Hodge structures of the fibers of $\pi$:
\begin{equation}\label{section-map}
\sigma(p)=H^{0}(X, K_{S})\in \mathcal{H}^{\V_{p}}.
\end{equation}
For any $\gamma^{PD} \in H_2(X,\Z)^\pi$ and $h \in \Z$ let $$\V_p(h,\gamma)=\{0\neq \beta \in \V_p|\beta^2=2h-2, i_* \beta^{PD}=\gamma^{PD}\},$$ and  $\operatorname{B}_p(h,\gamma)\subseteq \V_p(h,\gamma)$ be the subset containing $\beta \in \V_p(h,\gamma)$ where $\beta$ is a $(1,1)$ class on $S$ (the fiber over $p$). For any integer $h \in \Z $ the Noether-Lefschetz number $NL^\pi_{h,\gamma}$ is by definition the intersection number of $\sigma(C)$ with the $\pi$-relative Noether-Lefschetz divisor in $\mathcal{H}^\V$ associated to $h$ and $\gamma$.

$\operatorname{B}_p(h,\gamma)$ is finite by  \cite[Proposition 1]{a90}, and $\operatorname{B}(h,\gamma)=\cup_p \operatorname{B}(h,\gamma)\subset \V$ can be decomposed into $$\operatorname{B}_{I}(h,\gamma)\coprod \operatorname{B}_{II}(h,\gamma)$$ where the first component defines a finite local system $\epsilon:\operatorname{B}_{I}(h,\gamma)\to C$, and the second component is an isolated set. Let $\M_\epsilon$ be the type I component of $\M$ corresponding to the local system $\epsilon$, and let $$DT(X,\epsilon)=\deg([\M_\epsilon,E^\bullet]^{vir}) $$ be the contribution of this component to $DT(X, \ch)$. 
 By Propositions \ref{decompos} and \ref{decompos2} we have

\begin{align}\label{virtual}
[\M_\epsilon,E^\bullet]^{vir}&=c_{top}(\ob_\epsilon)\cap [\M_\epsilon,F_\epsilon^\bullet]^{vir} , \notag\\ [\M_\epsilon,F_\epsilon^\bullet]^{vir}&=\rho^* c_1(\K^\vee \otimes \omega_C) \cap [\M_\epsilon].\end{align} 

On the other hand, if $\alpha \in \operatorname{B}_{II}(h,\gamma)$ supported on the fiber $S$ is a result of a \emph{transversal intersection}  of $\sigma(C)$ with a Noether-Lefschetz divisor, then the corresponding connected component $\M_\alpha$ is an isolated type II component of $\M$ in the sense of Definition \ref{defn:comps} and is isomorphic to $\M(S,v)$ where $v$ is given as \eqref{mv}. 
By Proposition \ref{prop:decompos3} we have

\begin{equation}\label{virtual2} [\M_\alpha,E^\bullet]^{vir}=c_{top}(\mathcal{T}_{\M_\alpha})\cap [\M_\alpha].\end{equation} 
We let $$DT(X,\alpha)=\deg([\M_\alpha,E^\bullet]^{vir})$$ to be the contribution of $\M_\alpha$ to $DT(X,\ch)$.

\begin{proof}[Proof of Theorem \ref{thm:main formula}] We compare the contributions of $B_I$ and $B_{II}$ to 
$DT(X,\ch)$ and the Noether-Lefschetz numbers. Using the notation above 
\begin{align*}
DT(X,\epsilon)&=\int_{[\M_\epsilon]}c_{top}(\ob_\epsilon)\cdot \rho^* c_1(\K^\vee \otimes \omega_C)\\&=
\int_{[\M(S,v)]} c_{top}(\mathcal{T}_{\M(S,v)})\cdot \int_{B_{I}(h,\gamma)} \rho^* c_1(\K^\vee \otimes \omega_C)
\\&=e(\hilb^{h-r\ch_3-r^2}(S))\cdot \int_{B_{I}(h,\gamma)} \rho^* c_1(\K^\vee \otimes \omega_C),
\end{align*}  where the first equality is because of \eqref{virtual}, the second equality holds by Lemma \ref{ctop}, and the last equality is due to Proposition \ref{prop:def-inv}.

By virtue of \eqref{K} and \eqref{section-map}, and the argument in the proof of  \cite[Theorem 1]{a90}, $\int_{B_{I}(h,\gamma)} \rho^* c_1(\K^\vee \otimes \omega_C)$ gives the contribution of $B_{I}(h,\gamma)$ to $NL^\pi_{h,\gamma}$.

Next, suppose that $\alpha \in B_{II}(h,\gamma)$. Using the deformation invariance of DT invariants and the intersection numbers, we may assume that the corresponding component $\M_\alpha$ is an isolated type II component after possibly a small analytic perturbation of the section $\sigma$ (which locally turns a multiplicity $n$ intersection into $n$ transversal intersections).\footnote{
This can be done, because $\M_\alpha$ being a compact isolated component of the moduli space, is entirely contained in $\rho^{-1}(\Delta)$ where $\Delta\subset C$ is a small open disk.} 
Once this is done, the contribution of $\alpha$ to $NL^\pi_{h,\gamma}$ is exactly 1, and moreover we can use \eqref{virtual2} and Proposition \ref{prop:def-inv} to deduce $$DT(X,\alpha)=\int_{\M_\alpha}c_{top}(\mathcal{T}_{\M_\alpha})=e(\hilb^{h-r\ch_3-r^2}(S)).$$ 

The proof of theorem is completed by adding the correction term involving the Kronecker delta function to take into account the contributions of $\beta=0$ to $DT(X,\ch)$.
\end{proof}

By formula \eqref{symmetry}, Theorem \ref{thm:main formula} implies the following symmetry in the invariants:

\begin{cor} \label{cor:comp}  Suppose that $\ch$ is as in \eqref{ch}, and $k\in \Z, r' \in \Z_+$ are such that $$r' \mid r \ch_3 \gamma \cdot L+kF\cdot L^2/2+r^2$$ and $\ch'=(0,r'F,\gamma+kF\cdot L, (r\ch_3+\gamma \cdot L+kF\cdot L^2/2+r^2)/r'-r')$ satisfies ($\star$iv) in Section \ref{sec:intro}. Then $DT(X,\ch)=DT(X,\ch').$
\end{cor} \qed

\subsection{Nodal $K3$ fibrations} \label{sec:conifold}
In this section we consider the situation where $X$ is a projective nonsingular threefold and the morphism \eqref{equ:fibration} is smooth except at finitely many fibers; the smooth fibers are $K3$ surfaces, and the singular fibers are the contraction of nonsingular $K3$ surfaces at finitely many $-2$-curves. Therefore, finitely many fibers can have ordinary double point singularities (ODP).\footnote{Using the deformation invariance, finding the DT invariants of $X$, when the singularities of the fibers are of more general type of rational double points (RDP), may be reduced to this case.} Note that in this case, the condition ($\star$ii) in Section \ref{sec:intro} is automatically satisfied. We assume that the condition ($\star$iii) is also satisfied, and moreover  

\begin{assumption} \label{r=1}
We assume throughout this section that in the Chern character vector \eqref{ch}, $r=1$, so in particular ($\star$iv) in Section \ref{sec:intro} is satisfied for any choice of $\ch_2, \ch_3$.\footnote{In fact, if $r=1$, any pure coherent sheaf on $X$ 
with the Chern character as in \eqref{ch} is Gieseker stable.} 
\end{assumption}

Let $s_1,\dots, s_k\in X$ be the singular points of the fibers of $X$,  and assume that there are $k'>0$ singular fibers in $X$ over $c_1,\dots, c_{k'}\in C$. 
If $k'$ is even, define $c_0=c_1$ and if $k'$ is odd, define $c_0$ to be an arbitrary point of $C$ distinct from $c_1,\dots,c_{k'}$. Define $\epsilon: \widetilde{C}\to C$ to be the double cover of $C$ branched over the points $c_0,\dots,c_{k'}$. The fiber product $\epsilon^*(X)$ is a threefold with conifold singularities exactly at the inverse images of $s_i$'s. Denote by $\tX$ its small resolution with the exceptional nonsingular rational curves $e_1, \dots, e_k$, and let $\tpi:\tX\to \tC$ be the induced morphism. The normal bundle of $e_i$ in $\tX$ is isomorphic to $\O_{\P^1}(-1)\oplus \O_{\P^1}(-1)$ \cite{a121}. Moreover, we let $\epsilon_t: C_t\to C$ be a double cover of $C$ branched at $k+2\{k/2\} $ generic points of $C$ when $t\neq 0$ and $C_0=\tC$, and define $X_t=\epsilon^*_t(X)$.

Our plan is to relate the DT invariants of $\tX$ and $X_t$ which differ by the conifold transitions. As in GW theory \cite{a96}, \cite{a97} this can be done using the degeneration techniques. It is possible that $\tX$ is no longer projective in which case we use the modifications of Remark \ref{rem:nonproj} to define the DT invariants. See Appendix \ref{sec:app} for a review of the degeneration techniques in DT theory.

We start with the good degenerations (see Appendix \ref{sec:app})
\begin{equation*}
\tX \rightsquigarrow Y \bigcup_{D_1,\dots,D_k} \coprod_{i=1}^k \bP_{1} \quad \text{and}  \quad X_t \rightsquigarrow Y \bigcup_{D_1,\dots,D_k} \coprod_{i=1}^k \bP_{2}
\end{equation*} 
where the nonsingular projective threefold $Y$ is the blow up of $\epsilon^*(X)$ at the inverse images of $s_i$'s 
with the exceptional divisors $D_i\cong \P^1\times \P^1$, $$\bP_{1}\cong \P(\O_{\P^1}\oplus \O_{\P^1}(1)^2),$$ and $\bP_{2}$ is a nonsingular quadric in $\P^4$.
The first degeneration is the degeneration to the normal cone \cite{a35} in which $D_i \subset Y$ is attached to the divisor at infinity $H_1=\P(\O_{\P^1}(1)^2)$ in the $i$-th copy of $\bP_1$. The second degeneration is called the semistable reduction of a conifold degeneration \cite{a96} in which $D_i \subset Y$ is attached to a smooth hyperplane section $H_2$ in the $i$-th copy of $\bP_2$. For $i=1,2$ let $\bf{ch}^i$ be a Chern character vector on $\bP_i$ such that ${\bf{ch}}^i_0=0$ and ${\bf{ch}}^i=H_i$, and let $\M(\bP_i/H_i,\bf{ch}^i)$ be the corresponding moduli space of stable sheaves. 

\begin{lemma} \label{nothing on Pi} $\M(\bP_i/H_i,\bf{ch}^i)$ admits a perfect deformation-obstruction theory. 
\begin{proof}

The canonical bundle of $\bP_i$ is $K_{\bP_i}\cong -3H_i$. So for any geometric point $\F$ of the moduli space, we have $\ext^3(\F,\F)\cong \operatorname{Hom}(\F,\F,\otimes K_{\bP_i})=0$ by the stability of $\F$ and Serre duality.
\end{proof}
\end{lemma}

We denote by $\cX_1 \to \A^1$ and $\cX_2\to \A^1$ the total spaces of the first and second degenerations above. 
Let $\X_i\to \rC$ be the corresponding stack of expanded degenerations (see Appendix \ref{sec:app}).

Define open subsets $V_i=\tC \setminus \{\epsilon^{-1}(c_1),\dots ,\widehat{\epsilon^{-1}(c_i)},\dots ,\epsilon^{-1}(c_{k'})\}$, and let the invertible sheaves $L_t$ and $\tL$ on  $X_t$ and $\tX$ be as follows: $$L_t=\epsilon^*_t(L),\;\; \text{and} \;\; \tL=\widetilde{\epsilon}^*(L)-\sum_{i=1}^k\tD_i$$ where $\widetilde{\epsilon}:\tX\to X$ is the natural morphism, and  $\tD_i$ is a divisor on $\tX$ with $e_i\cdot \tD_i=-2$ and $e\cdot \tD_i =0$ for any other curve $e$ on $U_i=\tpi^{-1}(V_i)$. $L_t$ defines a polarization on $X_t$ and $\tL$ defines a polarization on $U_i$ for each $i=1,\dots,k'$. Let $F_t$,  $\widetilde{F}$, and $F'$ respectively be the class of fibers in the fibrations $X_t\to C_t$, $\tX\to \tC$, and $Y\to \tC$,
 and $\ch^t$, $\widetilde{\ch}$, and $\ch'$ be the Chern character vectors on respectively $X_t$, $\tX$, and $Y$ such that $$\ch^t_{0}=\widetilde{\ch}_{0}=\ch'_0=0, \quad \ch^t_{1}=F_t, \quad \widetilde{\ch}_{1}=\widetilde{F}, \quad \ch'_1=F'.$$
Note that we put no restrictions on $\ch'_k,\widetilde{\ch}_k,\ch^t_{k}$ for $k=2,3$.
Let $\M^{L_t}(X_t,\ch_t)$ and $\M^{\tL}(\tX,\widetilde{\ch})$ be the corresponding moduli spaces of stable sheaves. Note that $\tL$ is a polarization on $U_i$'s and $\M(\tX,\widetilde{\ch})$ is constructed by gluing the moduli spaces $\M(U_i,\widetilde{\ch})$ as in Remark \ref{rem:nonproj}.

Fix an invertible sheaf $\L_1$ on $\cX_1$ whose restriction to a general fiber is $\tL$ and to $Y$ is $$L':=b^*L-\sum_{i=1}^kD_i,$$ where $b:Y\to X$ is the natural morphsim. Similarly, fix a relatively ample invertible sheaf $\L_2$ on $\cX_2$ whose restriction to a general fiber is $L_t$ and to $Y$ is $L'$. The invertible sheaf $L'$ defines a polarization on $Y$. Let $\mathfrak{c}\mathfrak{h}^i$ for $i=1,2$ be a Chern character vector on $\X_i/\rC$ whose restriction to a general fiber over $\rC$ is respectively $\widetilde{\ch}$ and $\ch^t$. For  $1\le g \le k$, and $1\le  i_1,\dots, i_g \le k$, let $\ch^{i_1,\dots, i_g}$ be a Chern character vector on $Y$, with $\ch^{i_1,\dots, i_g}_0=0$ and $\ch^{i_1,\dots, i_g}_1=F'-D_{i_1}-\cdots-D_{i_g}$.
We let $$\M^{L'}(Y/D_1,\dots,D_k, \ch'), \quad \M^{L'}(Y/D_1,\dots,D_k, \ch^{i_1,\dots, i_g}),\quad\M(\X_i/\rC,\mathfrak{c}\mathfrak{h}^i)$$ be the corresponding relative moduli space of stable sheaves (Appendix \ref{sec:app}).

\begin{lemma} \label{ob theory on Y} Using the notation above, we have
\begin{enumerate} [i)] 
\item $\M(\X_i/\rC,\mathfrak{c}\mathfrak{h}^i)$ admits a perfect obstruction theory relative to $\rC$ and a virtual cycle of dimension 1.
\item $\M(Y/D_1,\dots,D_k,\ch')$ admits a perfect deformation-obstruction theory and a virtual cycle of dimension 0. 
\item $\M(Y/D_1,\dots,D_k,\ch^{i_1,\dots, i_g})$ admits a perfect deformation-obstruction theory with virtual cycle equal to zero.
\end{enumerate}
\end{lemma}
\begin{proof}

i) For any geometric point $\F$  of the moduli space $\ext^3(\F,\F)_0=0$ by Serre duality and the stability of $\F$ and noting that $\F\cong \F \otimes  \omega_{\cX_i/\A^1}$ where $\omega_{\cX_i/\A^1}$ is the relative dualizing sheaf. 

ii) For any geometric point $\F$  of the moduli space $\ext^3(\F,\F)_0=0$. This is because  we know that $K_Y\cong D_1+\cdots+D_k$ and $F\cdot D_i=0$, and since $c_1(\F)=F$ by assumption, Serre duality and the stability of $\F$ imply that $$\operatorname{Ext}^3(\F,\F)\cong \operatorname{Hom}(\F,\F)^\vee\cong \CC,$$ and hence $\operatorname{Ext}^3(\F,\F)_0=0$.

iii) For simplicity we assume $g=i_1=1$. 
By Serre duality $$\ext^3(\F,\F)=\Hom(\F,\F(D_1))^*$$ for any geometric point $\F$ of the moduli space. But $\F$ is the push forward of a rank 1 torsion free sheaf $\G$ on a nonsingular $K3$ surface, where $\G$ is an ideal sheaf of points $\I$ twisted by an invertible sheaf. Therefore,
$$\Hom(\F,\F(D_1))\cong \Hom_{K3}(\I,\I(\hat{e}_1)),$$ where $\hat{e}_1\cong \P^1$ is given as the intersection of $D_1$ with the proper transform of the fiber containing the curve $e_1$. Since $\hat{e}_1$ is a $-2$-curve on a $K3$ surface, we know that $H^0(\O(\hat{e}_1))=\CC$ which implies that $\ext^3(\F,\F)\cong \Hom_{K3}(\I,\I(\hat{e}_1))^*\cong \CC$, and hence $\ext^3(\F,\F)_0=0$ as required.

To prove the vanishing of the virtual cycle, we show that the virtual dimension is negative. By the Hirzebruch-Riemann-Roch calculation: \begin{align*} &\ex^0(\F,\F)-\ex^1(\F,\F)+\ex^2(\F,\F)-\ex^3(\F,\F)\\&=\int_{Y}((F-D_1)+\cdots)(-(F-D_1)+\cdots)(1-\frac{1}{2}\sum_i D_i+\cdots)\\&=(F-D_1)^2\cdot D_1/2=1,\end{align*} from which we get $\ex^1(\F,\F)-\ex^2(\F,\F)=-1.$ Hence the virtual cycle is zero.

\end{proof} 

Suppose that $P \in \Q[m]$ is a degree 2 polynomial with the leading coefficient equal to $FL^2/2$. We define $$DT_L(X,P)=\sum_{P_L(\ch)=P} DT(X,\ch),$$ where the sum is over all the Chern characters $\ch$ (of the form \eqref{ch} with $r=1$) with the Hilbert polynomial $P_L(\ch)=P$. 
Similarly, we can define $$DT_{L_t}(X_t,P), \quad DT_{\tL}(\tX,P),\quad DT_{L'}(Y,P),\dots.$$
Note that $\tL$ is a polarization only over $U_i$'s, but since any coherent sheaf under consideration is supported on a fiber $DT_{\tL}(\tX,P)$ is well-defined.

Now we are ready to state the conifold transition formula for our DT invariants:

\begin{prop}(Conifold Transition Formula) \label{prop:transition formula} Suppose that $P \in \Q[m]$ is a degree 2 polynomial with the leading coefficient equal to $FL^2/2$. Then,
$$DT_{L_t}(X_t,P)=DT_{\tL}(\tX, P).$$
\end{prop}
\begin{proof}
By Lemma \ref{ob theory on Y} $\M(\X_1/\fC ,\mathfrak{c}\mathfrak{h}^1)$ admits a perfect obstruction theory. 
Applying the degeneration formula \eqref{degform} in Appendix \ref{sec:app} which follows from the naturality of the virtual cycle $[\M(\X_1/\rC ,\mathfrak{c}\mathfrak{h}^1)]^{vir}$, and using part ii) of Lemma \ref{ob theory on Y} and Lemma \ref{nothing on Pi}, we can express $DT_{\tL}(\tX,P)$ in terms of the degrees of the product of the virtual cycles of the relative moduli spaces of $Y$ and $\bP_1$.  There are two possibilities for a geometric point $\F$ in the central fiber of $\M(\X_1/\fC,\mathfrak{c}\mathfrak{h}^1)$. Either $\F$ is completely supported on $Y$ (and possibly its degenerations) or there are some $i_1,\dots, i_g$ such that $$\ch'_1(\F|_Y)=F-D_{i_1}-\cdots-D_{i_g}, \quad \quad {\bf{ch}}^{i_j}_1(\F|_{\bP_1^{i_j}})=H_{i_j}$$ where $\bP_1^{i_j}$ is the $i_j$-th copy of $\bP_1$.  Only the former case contributes because of the vanishing of the virtual cycle proven in part iii) of Lemma \ref{ob theory on Y}. Therefore, $$DT_{\tL}(\tX,P)=DT_{L'}(Y/D_1,\dots,D_k,P).$$ 

Similarly, using the degeneration formula \eqref{degform} which follows from the naturality of the virtual cycle $[\M(\X_2/\rC ,\mathfrak{c}\mathfrak{h}^2)]^{vir}$, we can express $DT_{L_t}(X_t,P)$ in terms of the degrees of the product of the virtual cycles of relative moduli spaces of $Y$ and $\bP_2$.  Again by the same argument as in the last paragraph by distinguishing two similar cases and using Lemmas \ref{ob theory on Y} and \ref{nothing on Pi} we get $$DT_{L_t}(X_t,P)=DT_{L'}(Y/D_1,\dots,D_k,P).$$ 
Now the lemma follows from the last two identities.

\end{proof}

Now we choose $k'$ generic fibers $S_1,\dots, S_{k'}$ of $X\to \CC$. By our assumption $S_i$ is a nonsingular $K3$ surface. Let $\M(X/S_1,\dots, S_{k'},\ch)$ be the relative moduli space of stable sheaves with the Chern character $\ch$ (given in \eqref{ch}). By our conditions, one can see similar to the proof of Lemma \ref{ob theory on Y} that $\M(X/S_1,\dots, S_{k'}, \ch)$ admits a perfect obstruction theory and a virtual class of dimension zero.  

Let $X_i=S_i\times \P^1$. Then $X_i$ is a smooth $K3$-fibration over $\P^1$. We denote the class of the fibers by $F_i$, and $\ch^i$ a Chern character vector on $X_i$ such that $\ch^i_0=0$ and $\ch^i_1=F_i$ (and no restrictions on $\ch^i_2$ and $\ch^i_3$). Let $\M(X_i/S_i,\ch^i)$ be the corresponding relative moduli space of stable sheaves.
\begin{lemma} \label{K3-vanishing}
$DT(X_i/S_i, \ch^i)=0$.
\end{lemma}

\begin{proof}
$X_i$ is a smooth  $K3$ fibration, and Theorem \ref{thm:main formula} can be applied. Since $X_i$ is a trivial $K3$ fibration, $\M(X_i/S_i,\ch^i)$ has no type II components. In this case, $\K \cong \omega_C$, so $\rho^*c_1(\K^\vee\otimes \omega_C)=0$, and hence $DT(X_i,\ch^i)=0$. Next, by the degeneration formula \eqref{degform} applied to $X_i\rightsquigarrow X_i\coprod_{S_i} X_i$ and the irreducibility of the class $F_i$, we get $$2DT(X_i/S_i,\ch^i)=DT(X_i,\ch^i)$$ from which the result follows.
\end{proof}

\begin{lemma} \label{lem:rel=abs}
$DT(X/S_1,\dots, S_{k'}, \ch)=DT(X,\ch)$.
\end{lemma}
\begin{proof}
We consider the good degeneration $X\rightsquigarrow X \coprod_{S_1,\dots, S_r}  X_i$, and apply the degeneration formula \eqref{degform}. The irreducibility of the class $F$, and the vanishing result of  Lemma \ref{K3-vanishing} proves the claim. 
\end{proof}

 We use Lemma \ref{lem:rel=abs} to relate the DT invariants of $X$ to $X_t$.  To achieve this, we use the good degeneration of $X_t$ obtained by degenerating its base $\tC$ to two copies of $C$ followed by attaching two copies of $X$ along the generic fibers $S_1,\dots,S_{k'}$.  The degeneration formula \eqref{degform} then implies that $$DT_{L_t}(X_t, P)=2DT_L(X/S_1,\dots, S_{k'}, P).$$ This together with Lemma \ref{lem:rel=abs} and Proposition \ref{prop:transition formula} proves 
\begin{thm}\label{half-DT} Suppose that $P\in \Q[m]$ is a degree 2 polynomial with the leading coefficient equal to $FL^2/2$. Then
$DT_{L}(X, P)=DT_{L_t}(\tX,P)/2$.   
\end{thm} \qed

\begin{proof}[Proof of Theorem \ref{thm:K3}]

By \cite{a89} we know 
$$ \sum_{n\geq 0}e(\operatorname{Hilb}^{n}(K3))\cdot q^{n}=\prod_{n\geq 1}\frac{1}{(1-q^{n})^{24}},$$
so $\disp \sum_{n\geq 0}e(\operatorname{Hilb}^{n}(K3))\cdot q^{n-1}=\Delta(q)^{-1}$.  
 Combining this formula with the results of Theorems \ref{thm:main formula} and \ref{half-DT}, we can express the generating function of the DT invariants in terms of the product of two modular forms:
\begin{equation*} 
 Z(X,q)=\frac{\Phi^{\tpi}(q)}{2\Delta(q)}. 
 \end{equation*}

\end{proof}

\appendix

\section{Relative moduli space of sheaves and a degeneration formula} \label{sec:app}

In this appendix we review a part of the construction of Li and Wu in \cite{a87} and explain briefly how the moduli space of relative stable sheaves and the relative DT invariants can be defined in the special situation that is needed in this paper (Section \ref{sec:conifold}). 

Let $q:W\to \A^1$ be a \emph{good degeneration} of the projective threefolds, i.e.
\begin{enumerate}
\item $W$ is smooth, 
\item all the fibers except $\pi^{-1}(0)$ are smooth projective threefolds, 
\item  $\pi^{-1}(0)=W_1\cup_D W_2$ where $W_i$ is a smooth threefold, $D\subset W_i$ is smooth divisor, and $\pi^{-1}(0)$ is a normal crossing divisor in $W$.
\end{enumerate}
Li and Wu in \cite{a87} construct the Artin stack of expanded degenerations
$$
\xymatrix{\W \ar[r]^p \ar[d] &W\ar[d]^{q}\\
\rC \ar[r]^r& \A^1.}
$$
Away from $r^{-1}(0)$ the family $\W$ is isomorphic to the original family $$W \backslash \pi^{-1}(0)\to \A^1 \backslash 0.$$ The central fiber $\pi^{-1}(0)$ of the original family $W\to \A^1$ is replaced in $\W$ by a (non-disjoint) union over all $k$ of the $k$-step degenerations $$W[k]=W_1 \cup_D \P(\O\oplus N_D^\vee) \cup_D  \P(\O \oplus N_D^\vee)\cup_D \cdots \cup_D \P(\O \oplus N_D^\vee) \cup_D W_2$$ together with the automorphisms $\CC^{*k}$ induced from the $\CC^*$-action along the fibers of the standard ruled variety $\P(\O \oplus N_D^\vee)$. Similarly, for a pair of a smooth projective threefold $Y$ and a smooth divisor $D\subset Y$, the Artin stack of relative pairs 
$$
\xymatrix{\Y \ar[r]^p \ar[d] &Y\ar[d]\\
\fA \ar[r]& \text{spec } \CC}
$$

is defined in \cite{a87} using the $k$-step degenerations  
$$Y[k]=Y \cup_D \P(\O\oplus N_D^\vee) \cup_D  \P(\O \oplus N_D^\vee)\cup_D \cdots \cup_D \P(\O \oplus N_D^\vee)$$ together with the automorphisms $\CC^{*k}$ as above. 
We refer to the $j$-th ruled component (from left to right) of $Y[k]$ or $W[k]$ by $\Delta_j$ and for convenience $\Delta_0:=Y$ in the relative case and $\Delta_0:=W_1, \Delta_{k+1}:=W_2$ in the degeneration case; we also refer to the zero and infinity sections of the $j$-th ruled component by $D_j$ and $D_{j-1}$, respectively (if $k=0$ we take $D_0:=D$). Let $\pi_i:\Delta_i\to D$ be the natural projection.

\begin{defn}\label{def:rel}(\cite[Definition 3.1, 3.9, 3.12]{a87})
Let $\F$ be a coherent sheaf on a $\CC$-scheme $T$ of finite type, and suppose that $Z\subset T$ is a closed subscheme. $\F$ is called \emph{normal} to $Z$ if $\text{Tor}^{\O_T}_1(\F,\O_Z)=0$. A coherent sheaf $\F$ on $Y[k]$ or $W[k]$  is called \emph{admissible}  if $\F$ is normal to all $D_j$ for $j=0,\dots,k$. 

\end{defn}

Suppose that $\V$ is a locally free sheaf on $Y$ (respectively on $W$), and $\O(1)$ be an ample invertible sheaf on $Y$ (respectively $W\to \A^1$). Li and Wu construct the Quot schemes $\quot^{\V,P}_{Y/D}$ (respectively $\quot^{\V,P}_{\W/\rC}$)  
of quotients $\phi:p^*\V\to \F$ on $Y[k]$ (respectively on $W[k]$) for some $k$ satisfying 
\begin{enumerate}
\item $\F$ is addmissible,
\item $\phi$ has finitely many automorphisms covering the automorphisms induced by the $\CC^{*k}$-action on the target space,
\item The Hilbert polynomial of $\F$ with respect to $p^*\O(1)$ is $P$. 
\end{enumerate}
Moreover, they show that $\quot^{\V,P}_{Y/D}$ (respectively $\quot^{\V,P}_{\W/\rC}$) is a separated, proper over (respectively separated, proper over $\A^1$), Deligne-Mumford  (DM) stack of finite type (\cite[Theorems 4.14 and 4.15]{a87}). 

Even though the line bundle $p^*\O(1)$ is used for distingushing the components of $\quot^{\V,P}_{Y/D}$ and $\quot^{\V,P}_{\W/\rC}$, for the construction of the moduli space of sheaves on $\Y/\fA$ and $\W/\rC$ we will need to modify it so that it restricts to a $\Q$-ample divisor on the support of our 2-dimensional sheaves in $Y[k]$'s or $W[k]$'s. We here treat the relative case $\Y/\fA$. The degeneration case $\W/\rC$ can be handled similarly. 

First recall the construction of $\Y\to \fA$. It is constructed as the limit of stack quotient of $Y(k)\to \A^k$ by certain equivalence relations (see \cite{a87}) where $Y(k)$ is defined inductively by $$(Y(0),D(0)):=(Y,D),\quad \quad Y(k):=\operatorname{Bl}_{D(k-1)\times 0}(Y(k-1)\times D(k-1))\quad k\ge 1,$$ with $D(k)$ being the prober transform of $D(k-1)\times \A^1$. The morphism $p:\Y\to Y$ is induced from the natural projections $p:Y(k)\to Y$ that we also denote by the same symbol $p$. 

For any $k\ge0$, we define\footnote{We have dropped the obvious pullback symbols from the notation.}
$$\tilde H_k:=mH-\sum_{i=1}^k g_i \Delta(i)$$ where $H:=p^*\O(1)$, $\Delta(i):=\Delta_i\times \A^{k-i}$ and \begin{equation}\label{gis}g_1=1,\; g_2=3/2,\; g_3=7/4,\; \dots, g_k=2-1/2^{k-1}.\end{equation}
One can check that $\tilde H_k$ is invariant under all equivalence relations on $Y(k)$ and hence descends to a relative $\Q$-Cartier divisor $\tilde H$ on $\Y\to \fA$.

\begin{lemma} \label{nakmoi}
Let $\F$ be an admissible pure 2-dimensional sheaf on $Y[k]$ so that $\F|_{D_i}$ is pure 1-dimensional, and the support of $\F|_{\Delta_i}$ is integral for any $i=0,\dots, k$, then $\tilde H |_{\supp(\F)}$ is very ample for $m\gg 0$.
\end{lemma}
\begin{proof}
First of all, if $C$ is any curve in $\Delta_0$ or in $D_i$ for $i>1$, and also if $G$  is a 2-dimensional subvariety of $\Delta_0$ or is one of $D_i$ for $i>1$ then we have $$\tilde H_k\cdot C>0,\quad \tilde H_k^2\cdot G>0$$ all because $m\gg 0$, and $H$ is ample on $Y$.

Next suppose $f_i$ is the class of a fiber of $\Delta_i$ for $i>0$ and $G_i:=\pi_i^{-1}(C)$ where $C\subset D\subset X$ is a curve then using the relations 

$$\Delta_i^2=-D_{i-1}-D_i \;\;\; i=1,\dots,k,\quad \Delta_k^2=-D_{k-1},$$ in the total space of $Y(k)$, and the fact that $m\gg0$ we can see that 
$$\tilde H_k\cdot f_i>0,\quad \tilde H_k^2\cdot G_i>0$$  are satisfied provided that $$2g_1>g_2,\;\; 2g_2>g_1+g_3,\dots, 2g_{k-1}>g_{k-2}+g_k,\;\; g_k>g_{k-1},\quad g_0=0, \quad k\ge 1.$$ It is now easy to see that the choice given in \eqref{gis} satisfies these inequalities.
\end{proof}





\begin{lemma} \label{Hstab1}\begin{enumerate}[i)]
\item Let $\F$ be as in Lemma \ref{nakmoi} so that $$\gamma:=\ch_2(\F), \quad \supp(\F|_{\Delta_i})=\pi_i^{-1}(\pi_i(\supp(\F|_{D_{i-1}}))),$$ and suppose $\F$ is $\tilde H$- semistable for $m\gg 0$. Then $\F$ is $\tilde H$-stable.
\item Let $\F$ be exactly as in part i) except that instead for some $1<j\le k$ the class of $\supp(\F|_{\Delta_j})$ is the same as the class of $D_{j-1}$. Then $\F$ is $\tilde H$-stable.
\end{enumerate}
\end{lemma}
\begin{proof}
We prove part i) for $k=1$. The other cases are similar.
 Let $$\Delta'_i:=\Delta_i\cap \supp(\F),\quad \quad  D'_i:=D_i\cap \supp(\F) \quad  \quad i=0, 1.$$ 

The $\tilde H_1$-semistability is implies  $$\frac{(\gamma_1-D'_0)\cdot \tilde H_1}{\Delta'_1\cdot \tilde H_1^2} \le\frac{\gamma\cdot \tilde H_1}{(\Delta'_0+\Delta'_1)\cdot \tilde H_1^2}\le \frac{\gamma_1\cdot \tilde H_1}{\Delta'_1\cdot \tilde H_1^2}.$$  Suppose that $$r:=\gamma \cdot D_0\in  \frac 1 2 \Z, \quad l:=\Delta'_0\cdot H^2\in \Z,\quad u:=D'_0\cdot  H\in \Z, \quad d:=D\cdot D'_0\in \Z.$$ Then, by the condition on $\supp(\F|_{\Delta_i})$ and the admissibility $\ch_2(\gamma |_{\Delta_1})=aD'_0+rf_1$ for some $a\in \frac 1 2 \Z$. We can compute 
$$(\Delta'_0+\Delta'_1)\cdot \tilde H_1^2=m^2l,\quad \Delta'_1\cdot \tilde H_1^2=2mu-d, \quad \gamma_1\cdot \tilde H_1=a(mu-d)+r.$$ Taking $v:=\gamma\cdot \tilde H_1$  the inequalities above are equivalent to $$z\le a\le z+1 \quad \text{where} \quad z=\frac{(2mu-d)v-m^2lr}{m^2l(2mu-d)} .$$ Now the denominator and numerator of $z$ are respectively of degrees 3 and at most 2 in $m$ respectively, thus, for $m\gg0$ the inequalities above must be strict.
\end{proof}

\begin{lemma} \label{refl}
Suppose $\G$ is a pure 2-dimensional sheaf in a nonsingular projective threefold $X$. Suppose that $D\subset X$ is a nonsingular divisor and both $\G$ and $\G^{**}/\G$ are normal to $D$. Then, $\G|_D$ is a pure 1-dimensional sheaf. 
\end{lemma}
\begin{proof}
Here as in \cite[Def. 1.1.7]{a9}, $\G^*=\Ext^1(\G,\omega_X)$. By \cite[Prop. 1.1.10]{a9} $\G^{**}$ is reflexive, the natural map $\G\to \G^{**}$ is injective, and $$\G^{**}/\G\cong \Ext^2(\G,\omega)$$ is 0-dimensional. The normality of $\G$ implies that $\supp(\G)$ is transversal to $D$ and so are $\supp(\G^*)$ and $\supp(\G^{**})$. The normality of $\G^{**}/\G$ implies that $$\supp(\G^{**}/\G)\cap D=\emptyset.$$

Now by  \cite[Lem. 1.1.13]{a9} $\G^{**}|_{D}$ is pure 1-dimensional, so restricting the natural short exact sequence $$0\to \G\to \G^{**}\to \G^{**}/\G\to 0$$ to $D$ and using the above facts about the support of $\G^{**}/\G$, we conclude $\G|_D\cong \G^{**}|_D$ and hence is pure 1-dimensional.
\end{proof}

Fix a degree 2 polynomial $P$ and $N\gg 0$. We will consider the moduli space of the following objects:
\begin{enumerate}
\item $\F$ is a pure 2-dimensional sheaf on $Y[k]$ for some $k$ so that the Hilbert polynomial of $\F$ with respect to $H=p^*\O(1)$ is $P$.
\item $\F$ and $\F^{**}/\F$ are admissible.
\item $\F$ is semistable with respect to $\tilde H$ for $m\gg 0$.
\item There exists a surjection $$\V=\oplus_{i=1}^{P(N)}p^*\O(-N)\xrightarrow{s} \F$$ such that the pair $(\F,s)$ has only finitely many automorphism covering $\CC^{*k}$-automorphisms of $Y[k]$. 
\end{enumerate} 

Now for any $S$-family $\mathcal Y \in \Y(S)$ of relative pairs, an $S$-valued point of the moduli stack $\mathfrak{M}(Y/D,P)$ is by definition an $S$-flat family $\mathscr F$ of sheaves whose restriction to each fiber of $\mathcal Y_s$ over a closed point $s\in S$ satisfies conditions (1)-(4) above. A 1-morphism between two $S$-valued points $(\mathscr F, \mathcal Y)$ and $(\mathscr F', \mathcal Y')$ is a 1-morphism $\sigma:\mathcal Y\to \mathcal Y'$ in $\Y(S)$ such that $\mathscr F\cong \sigma^*\mathscr F'$.

 A smooth atlas for this moduli stack is  obtained by \begin{equation} \label{atlas}\coprod_{k\ge 0}\left[\quot^{p^*\V,P}(Y(k)/\A^k)^{\circ}/\CC^{*k} \right]\to \mathfrak{M}(Y/D,P),\end{equation} where $\circ$ stands for the open subset of the usual Quot scheme $\quot^{p^*\V,P}(Y(k)/\A^k)$ of quotients satisfying  conditions (1)-(4). The openness of condition (3) is standard, and the openness of conditions (2) and (4) was proven by \cite{a87}. For each $k$ the quotient stack is a Deligne-Mumford stack by the condition (4) as in \cite{a87}. The arrow in \eqref{atlas} is induced by the universal family over $\quot^{p^*\V,P}(Y(k)/\A^k)^{\circ}$. By the standard arguments then it follows that $\mathfrak{M}(Y/D,P)$ is an Artin stack of finite type. Here, we are also using the boundedness result of \cite{a87} as well as the usual boundedness of the semistable sheaves with fixed Hilbert polynomials.
 
 \begin{remark}
It can be seen that $ \mathfrak{M}(Y/D,P)$ is the stack quotient (see \cite{R05}) of Li-Wu Quot stack
$$ \mathfrak{M}(Y/D,P)\cong \quot^{\V,P}_{Y/D}/GL(P(N)).$$  
\end{remark}
 
Since any coherent sheaf has automorphisms which act as multiples of the identity, one can see that $\mathfrak{M}(Y/D,P)$ is $\CC^*$-gerb over an Artin stack that we denote by $\mathcal{M}(Y/D,P)$. In the case that there are no strictly semistable sheaves (with respect to $\tilde H$), by the condition (4) above and the argument in \cite{a87}, $\mathcal{M}(Y/D,P)$ is a Deligne-Mumford stack.

By the virtue of Assumption \ref{r=1}, one can check that in all the degeneration situations considered in this paper we can arrange for the following assumption to be satisfied. The key point is that under $r=1$ condition in Section \ref{sec:conifold}, any pure coherent sheaf with reduced irreducible support is Gieseker stable with respect to any polarization.

\begin{assumption} \label{relmod}
Any closed point of the moduli stack $\mathcal{M}(Y/D,P)$ corresponds to a coherent sheaf $\F$ on $Y[k]$ for some $k$ satisfying  the conditions in Lemmas \ref{nakmoi} and \ref{Hstab1}.
\end{assumption}

Suppose that  Assumption \ref{relmod} is satisfied. First of all, by Lemma \ref{Hstab1} and the discussion above $\mathcal{M}(Y/D,P)$ is a Deligne-Mumford stack. Note that since we require the Hilbert polynomial with respect to $H$ to be $P$, by admissibility of $\F$, the support of $\F$ can be of one of the forms in Lemma \ref{Hstab1}. 

Moreover, by Assumption \ref{relmod}, Li-Wu's proofs of separatedness/properness will go through to prove the separatedness/properness of our moduli stack $\mathcal{M}(Y/D,P)$. In fact, the main ingredients of Li-Wu proofs is first, the error function by which they measure the deviation from the admissibility, and second, the separatedness/ properness of the usual Grothendieck's Quot schemes. They show that after finitely many steps of modifications of a given non-admissible sheaf, their error function eventually vanishes and hence they attain admissibility. We use exactly their error function to control the admissibility, but instead of the quot schemes, we use the separatendness/properness of the usual moduli space of stable sheaves, noting that again by Lemma \ref{Hstab1}, we allow no $\tilde H$-strictly semistable sheaves with the given Hilbert polynomial $P$.

Suppose for a given coherent sheaf corresponding to a closed point of $\mathcal{M}(Y/D,P)$, the Hilbert polynomial of $\F|_{D_k}$ with respect to $H$ is $P_0$ (it is a degree 1 polynomial). By Lemma \ref{refl}, $\F|_{D_k}$ is a pure 1-dimensional sheaf. Suppose now that for any pure coherent sheaf $\F_0$ with Hilbert polynomial $P_0$ of our interest, we know that $\supp(\F_0)$ is reduced and irreducible; so $\F_0$ is stable with respect to any polarization and hence corresponds to a closed point of the moduli space of stable of $H|_D$-stable sheaves, denoted by $\M(D,P_0)$. Therefore, the restriction to the relative divisor defines a natural morphism \begin{equation} \label{relmap}\mathcal{M}(Y/D,P)\to \mathcal{M}(D,P_0).\end{equation}

%

Similarly, using  parallel Assumption \ref{relmod} and Lemmas   \ref{nakmoi}-\ref{refl} in the degeneration case, we can construct the Deligne-Mumford moduli stack $\M(\W/\rC,P)$ which is a finite type, separated and proper over $\fC$. 

Since we do not allow strictly semistable sheaves, in particular any closed point $\F \in \M(Y/D,P)$ (respectively $\M(\W/\rC,P)$) is simple, so if we additionally have  $\ext^3(\F,\F)_0=0$ then, by the standard arguments \cite{a20,a10, a84}, there is a perfect obstruction theory relative to the base $\fA$ (respectively $\rC$) given by 
\begin{equation} \label{equ:rel theory}\left(\tau^{[1,2]}R\pi_{*}(R\mathscr{H}om(\mathbb{F},\mathbb{F}))\right)^{\vee}[-1],
\end{equation} where $\FF$ is the universal sheaf (see the footnote on page 5) and $\pi$ denotes the projections $$\M(Y/D,P)\times_{\fA} \Y\to \M(Y/D,P), \quad  \M(\W/\rC,P)\times_{\rC} \W \to \M(\W/\rC,P).$$
Hence we get the virtual cycles  $$[\M(Y/D,P)]^{vir}\in A_{n}(\M(Y/D,P)),\quad  \quad [\M(\W/\rC,P)]^{vir}\in A_{n+1}(\M(\W/\rC,P)),$$ where $n$ is the rank of the obstruction theory above. 
From now on we assume that $n=0$, in which case, we can define the \emph{relative DT invariant} $$DT(Y/D, P)=\deg \;[\M(
Y/D,P)]^{vir}.$$

By naturality of the virtual cycle, the restriction of  $[\M(\W/\rC,P)]^{vir}$ to a general fiber $W_t$ of $q:W\to \A^1$ is $[\M(W_t,P)]^{vir}\in A_0(\M(W_t,P))$, and  
the degeneration formula for DT invariants can be expressed as  \begin{equation}\label{degfor}\deg [\M(W_t,P)]^{vir}= \deg [\M(\W^\dagger_0/\rC^\dagger_0,P)]^{vir},\end{equation}
where 
$$
\xymatrix{\W_0^\dagger \ar[r] \ar[d] &\W\ar[d]\\
\rC_0^\dagger \ar[r]& \rC}
$$
is the \emph{stack of node marking objects} in $\W_0$ as defined in \cite[Section 2.4]{a87} and $\M(\W^\dagger_0/\rC^\dagger_0,P)\subset \M(\W/\rC,P)$ is the corresponding open and closed subset. 

Now suppose that the moduli space $\M(D,P_0)$ is nonsingular. This is well-known to be the case if $D$ is a semi-Fano surface, e.g. $K3$ or $\P^2$. In this situation one can expand the right hand side of \eqref{degfor} into a more useful formula. 
Given the triple of Hilbert polynomials $\eta:=(P_1,P_0,P_2)$ of degrees $(2, 1, 2)$ satisfying $P=P_1-P_0+P_2$, we can form the following Cartesian diagram  

$$
\xymatrix{\M^{P_1, P_0, P_2} \ar[r] \ar[d] & \M(\W_0^\dagger/\rC_0^\dagger,P)\ar[d]\\
\rC^{\dagger, P_1, P_0, P_2}_0 \ar[r]& \rC}
$$ where the superscripts $P_i$ stand for the substack of $\rC$ decorated by the given Hilbert polynomials (see \cite{a87}). Similarly, we will consider the substacks $ \fA^{P_1,P_0}, \fA^{P_2,P_0}$ of $\fA$.
We have the commutative diagram 

$$
\xymatrix{\M(W_1/D,P_1) \times_{\M(D,P_0)} \M(W_2/D,P_2) \ar[r]^-{\Phi}  \ar[d] &\M^{P_1, P_0, P_2} \ar[d]\\
\fA^{P_1,P_0}\times \fA^{P_2,P_0} \ar[r]^-{\cong}& \rC^{\dagger,P_1, P_0, P_2}_0}
$$ in which $\Phi$ is an isomorphism for the same reason as in \cite[Theorem 5.27]{a87}. Using this isomorphism, we can form the following Cartesian diagram
 
\begin{equation} \label{diagM}
\xymatrix{\M^{P_1, P_0, P_2} \ar[r]  \ar[d] &\M(W_1/D,P_1)\times \M(W_2/D,P_2) \ar[d]\\
\M(D,P_0) \ar[r]^-{\Delta}& \M(D,P_0)\times \M(D,P_0)}
\end{equation}
in which the vertical morphisms are defined by the restriction to the relative divisors as in \eqref{relmap}.

Suppose that $\FF, \FF_i, \FF_D$ are the universal families on $$\M^\eta \times_{\rC_0^{\dagger,\eta}} \W_0^{\dagger,\eta},\quad \M(W_i/D,P_i)\times_{\fA^{P_i,P_0}}\Y^{P_i,P_0},\quad  \M(D,P_0)\times D,$$ respectively and $\pi$ denotes the projections to the first factors in all the above products. Consider the following perfect obstruction theories
\begin{align}\label{perob}
F^\bullet:=\left(\tau^{[1,2]}R\pi_{*}(R\mathscr{H}om(\mathbb{F},\mathbb{F}))\right)^\vee[-1] &\xrightarrow{f} \mathbb L^\bullet_{\M^\eta/\rC^{\dagger,\eta}_0},\\ \notag
F_i^\bullet:=\left(\tau^{[1,2]}R\pi_{*}(R\mathscr{H}om(\mathbb{F}_i,\mathbb{F}_i))\right)^\vee[-1] &\xrightarrow{f_i} \mathbb L^\bullet_{\M(W_i/D,P_i)/{\fA^{P_i,P_0}}},\quad i=1,2,\\ \notag
F_D^\bullet:=\left(R\pi_{*}(R\mathscr{H}om(\mathbb{F}_D,\mathbb{F}_D))_0\right)^\vee &\xrightarrow{\cong} \mathbb {L}^\bullet_{\M(D,P_0)}\cong \Omega_{\M(D,P_0)},
\end{align} where in the last we are assuming that $D$ is semi-Fano (and hence smoothness of $\M(D,P_0)$).  
These all fit into the following commutative diagram in which the rows are exact triangles:
$$
\xymatrix{F_D^\bullet \ar[r]  \ar[d]^-\cong & \oplus_{i=1}^2 F_i^\bullet \ar[r] \ar[d]^-{(f_1,f_2)} & F^\bullet \ar[d]^-f\\
 L^\bullet_{\M^\eta/\M(W_1/D,P_1)\times \M(W_2/D,P_2)}\ar[r]& \oplus_{i=1}^2 L^\bullet_{\M(W_i/D,P_i)/{\fA^{P_i,P_0}}}\ar[r] &    L^\bullet_{\M^\eta/\rC^{\dagger,\eta}_0}. }
$$ Here the exactness of the upper row follows from the natural short exact sequence $$0\to \FF\to \FF_1\oplus\FF_2\to \FF_D\to 0$$ and the admissibility conditions, and the bottom row is the natural exact triangle of the cotangent complexes. The first column is an isomorphism because of the last line of \eqref{perob}, diagram \eqref{diagM} and the isomorphism 
$$\Omega_{\M(D,P_0)}\cong \mathbb L^\bullet_{\M(D,P_0)/\M(D,P_0)\times \M(D,P_0)}[-1].$$

Now the last diagram together with the argument in \cite[Section 6]{a87} (and also \cite{a84}), implies the following degeneration formula

\begin{equation} \label{degform}\deg[\M(W_t,P_t)]^{vir}=\sum_{P=P_1-P_0+P_2}\deg\left( [\M(W_1/D,P_1)]^{vir}\times [\M(W_2/D,P_2)]^{vir}\right)\end{equation} where the sum is over all possible splitting of the Hilbert polynomial $P$ on $W_t$ into the Hilbert polynomials $P_1,P_0,P_2$ on $W_1,D,W_2$ as in the above discussion.

\noindent {\tt{amingh@math.umd.edu}},\quad 
\noindent {\tt{University of Maryland}} \\
\noindent {\tt{College Park, MD 20742-4015, USA}} \\\\
\noindent {\tt{artan@qgm.au.dk}},\quad 
\noindent{\tt{Centre for Quantum Geometry of Moduli Spaces, Aarhus University, Department of Mathematics
Ny Munkegade 118, building 1530, 319, 8000 Aarhus C, Denmark}}\\\\
\noindent{\tt{artan@cmsa.fas.harvard.edu, Center for Mathematical Sciences and\\ Applications, Harvard University, Department of Mathematics, 20 Garden Street, Room 207, Cambridge, MA, 02139}}\\\\
\noindent{\tt{artan@cmsa.fas.harvard.edu, National Research University Higher School of Economics, Russian Federation, Laboratory of Mirror Symmetry, NRU HSE, 6 Usacheva str., Moscow, Russia, 119048}}\\\\


\begin{thebibliography}{12}

  
\bibitem[Ati58]{a121}
{M. F. Atiyah}, \emph{{On Analytic Surfaces with Double Points}}, {Proc. Royal
  Soc. London} \textbf{{247 (1249), pp. 237--244}} (1958).
  
  
\bibitem[Beh09]{a1}
{K.~Behrend}, \emph{{{Donaldson}--{Thomas} invariants via microlocal geometry}},
  {Annals of Math.} \textbf{170, pp. 1307--1338} (2009).

\bibitem[BF97]{a2}
{K.~Behrend, B.~Fantechi}, \emph{{The intrinsic normal cone}}, {Invent.
  Math.} \textbf{128, pp. 45--88} (1997).


\bibitem[Bor98]{a125}
{R.~Borcherds}, \emph{{Automorphic forms with singularities on Grassmannians}},
  {Invent. Math.} \textbf{132, pp. 491--562} (1998).

\bibitem[Bor99]{a116}
\bysame, \emph{{The Gross-Kohnen-Zagier theorem in higher dimensions}}, {Duke J.
  Math.} \textbf{97, pp. 219--233} (1999).

\bibitem[DM07]{DM07}
F.~Denef, G.~Moore, \emph{{Split states, entropy enigmas, holes and halos}}, {J.  High Energy Phys.} \textbf{11 (129)} (2011).



\bibitem[Ful98]{a35}
{W.~Fulton}, \emph{Intersection theory}, Springer (1998).

\bibitem[GSY07]{a78}
D.~Gaiotto, A.~Strominger, and X.~Yin, \emph{{The M5-Brane Elliptic Genus:
  Modularity and BPS States}}, {J. of high energy phys.} \textbf{8 (070), pp. 1--17} (2007).
  

\bibitem[GY07]{a95}
{D.~Gaiotto and X.~Yin}, \emph{{Examples of M5-Brane Elliptic Genera}},  {J. of high energy phys.} \textbf{11 (004), pp. 1--12} (2007).



\bibitem[GST14]{GST14}
A.~Gholampour, A.~Sheshmani, R.~Thomas, \emph{Counting curves on surfaces in Calabi-Yau 3-folds}, {Math. Annalen.} \textbf{360, pp. 67--78} (2014).


  
 \bibitem[Got90]{a89}
{L. G\"{o}ttsche}, \emph{{The Betti numbers of the Hilbert scheme of points on a
  smooth projective surface}}, {Math. Ann.} \textbf{286, pp. 193--207} (1990).
 


\bibitem[HL97]{a9}
\bysame, \emph{The geometry of moduli spaces of sheaves}, {Cambridge University
  press} (1997).

\bibitem[HT10]{a10}
{D.~Huybrechts and R.~P. Thomas}, \emph{Deformation--Obstruction theory for
  complexes via {A}tiyah and {Kodaira}--{Spencer} classes}, {Math. Ann.}
  \textbf{346 (3), pp. 545--569} (2010).


\bibitem[KY00]{a79}
{T.~Kawai and K.~Yoshioka}, \emph{{String partition function and infinite
  products}}, \textbf{arXiv:hep-th/0002169} (2000).


\bibitem[KKRS04]{a122}
{A.~Klemm, M.~Kreuzer, E.~Riegler, and E.~Scheidegger}, \emph{{Topological string
  amplitudes, complete intersections Calabi-Yau spaces, and threshold
  corrections}}, \textbf{arXiv:hep-th/0410018} (2004).

\bibitem[KL10]{a137}
{Y.~Kiem and J.~Li}, \emph{{Localizing virtual cycles by cosections}},
  \textbf{arXiv:1007.3085} (2010).

\bibitem[KY11]{a139}
M.~Kimura and K.~Yoshioka, \emph{{Birational Maps of Moduli Spaces of Vector
  Bundles on K3 Surfaces}}, {Tokyo J. of Math.} \textbf{34, No. 2
 , pp. 473--491.} (2011).



\bibitem[Kre99]{a111}
{A.~Kresch}, \emph{{Cycle groups for Artin stacks}}, {Invent. Math.} \textbf{138,
  pp. 495--536} (1999).


\bibitem[KM90]{a117}
{S.~Kudla and J.~Millson}, \emph{{Intersection numbers of cycles on locally
  symmetric spaces and Fourier coefficients of holomorphic modular forms in
  several complex variables}}, {Pub. IHES}, \textbf{71, pp. 121--172} (1990).


\bibitem[Leh98]{a138}
{M.~Lehn}, \emph{{On the cotangent sheaf of Quot-schemes}}, {Internat. J. of Math.}
  \textbf{9, No. 4, pp. 513--522.} (1998).



\bibitem[LR01]{a97}
{A.~Li and Y.~Ruan}, \emph{{Symplectic surgery and Gromov-Witten invariants of
  Calabi-Yau 3-folds}}, {Invent. Math.} \textbf{45 (1), pp.
  151--218} (2001).

\bibitem[LT98]{a14}
{J.~Li and G.~Tian}, \emph{Virtual moduli cycles and Gromov-Witten invariants of
  algebraic varieties}, {J. Amer. Math. Soc.} \textbf{11, pp. 119--174} (1998).

\bibitem[LW11]{a87}
{J.~Li and B.~Wu}, \emph{{Good degeneration of Quot-schemes and coherent
  systems}}, \textbf{arXiv:1110.039} (2011).

\bibitem[LY04]{a96}
{C.~Liu and S.~Yau}, \emph{{Extracting Gromov--Witten invariants of a conifold
  from semi-stable reduction and relative GW invariants of pairs}},
 \textbf{arXiv:math/0411038} (2004).

\bibitem[MP06b]{a120}
{D.~Maulik and R.~Pandharipande}, \emph{{New calculations in Gromov--Witten
  theory}}, {Pure Appl. Math.}  (2006).

\bibitem[MP13]{a90}
{D.~Maulik and R.~Pandharipande}, \emph{{Gromov--Witten theory and
  Noether-Lefschetz theory}}, { Clay Math. Proc.}, \textbf{8}  (2013).

\bibitem[MPT10]{a84}
{D.~Maulik, R.~Pandharipande, and R.~P. Thomas}, \emph{{Curves on K3 surfaces and
  modular forms}}, {J. of Topol.} \textbf{3, pp. 937--996} (2010).

\bibitem[Muk84]{a113}
{S.~Mukai}, \emph{{Symplectic structure of the moduli space of sheaves on an
  abelian or K3 surface}}, {Invent. Math.} \textbf{77 (1), pp.
  101--116} (1984).

\bibitem[OSV01]{a100}
{H.~Ooguri, A.~Strominger, and C.~Vafa}, \emph{{Black hole attractors and the
  topological string}}, {Physics Review D.} \textbf{70 (10), pp. 106--119 } (2001).


\bibitem[PT09]{a17}
{R.~Pandharipande and R.~P. Thomas}, \emph{{Curve} counting via stable pairs in
  the derived category}, {Invent. Math.} \textbf{178, 407--447} (2009).

\bibitem[R05]{R05} M.~Romagny,  \emph{{Group} actions on stacks and applications}, {Mich. Math. J.} \textbf{53,pp 209--236}, (2005).

\bibitem[Tho00]{a20}
{R.~P.~Thomas}, \emph{A holomorphic {Casson} invariant for {Calabi}-{Yau}
  3-folds, and bundles on {K3} fibrations}, {J. Differential Geom.} \textbf{54, pp. 367--438}
  (2000).

\bibitem[VW94]{a136}
{C.~Vafa and E.~Witten}, \emph{{A strong coupling test of S-duality}}, {J. Nucl.
  Phys.} \textbf{B431, pp. 3--77} (1994).

\end{thebibliography}
\end{document}